\documentclass[11pt,reqno]{amsart}
\usepackage{amssymb,amsmath,tabularx}
\usepackage{amsthm,verbatim}
\usepackage[all]{xy}
\usepackage[bookmarks=true]{hyperref}

\newtheorem{thm}{Theorem}
\newtheorem{lem}[thm]{Lemma}

\newtheorem{cor}[thm]{Corollary}
\theoremstyle{definition}

\newcommand{\GL}{\mathrm{GL}}

\newcommand{\Lie}{\operatorname{\mathrm{Lie}}}

\newcommand{\sym}[1]{\mathfrak{S}_{#1}}
\newcommand{\Ind}{\operatorname{Ind}}
\newcommand{\sgn}{\operatorname{sgn}}
\newcommand{\E}{\boldsymbol{/\!\backslash}}

\begin{document}
\title{The Schur functor on tensor powers}

\author{Kay Jin Lim}
\author{Kai Meng Tan}
\address{Department of Mathematics, National University of Singapore, Block S17, 10 Lower Kent Ridge Road, Singapore 119076.}
\email[K. J. Lim]{matlkj@nus.edu.sg}
\email[K. M. Tan]{tankm@nus.edu.sg}

\date{February 2011}

\thanks{Supported by MOE Academic Research Fund R-146-000-135-112.}


\begin{abstract}
Let $M$ be a left module for the Schur algebra $S(n,r)$, and let $s \in \mathbb{Z}^+$.  Then $M^{\otimes s}$ is a $(S(n,rs), F\sym{s})$-bimodule, where the symmetric group $\sym{s}$ on $s$ letters acts on the right by place permutations.  We show that the Schur functor $f_{rs}$ sends $M^{\otimes s}$ to the $(F\sym{rs},F\sym{s})$-bimodule $F\sym{rs} \otimes_{F(\sym{r} \wr \sym{s})} ((f_rM)^{\otimes s} \otimes F\sym{s})$.  As a corollary, we obtain the image under the Schur functor of the Lie power $L^s(M)$, exterior power $\E^s(M)$ of $M$ and symmetric power $S^s(M)$.
\end{abstract}

\maketitle

\section{Introduction} \label{S:intro}

The representations of general linear groups and symmetric groups are classical objects of study. Following the work by Schur in 1901, there is an important connection between the polynomial representations of general linear groups and the representations of symmetric groups via the Schur functor. In this short article, we examine the images of tensor powers, Lie powers, symmetric powers and exterior powers under the Schur functor.

Our motivation comes from our study of the Lie module $\Lie(s)$ of the symmetric group $\sym{s}$ on $s$ letters.  This may be defined as the left ideal of the group algebra $F\sym{s}$ generated by the Dynkin-Specht-Wever element
$$
\upsilon_s = (1-c_2)\dotsm (1-c_s),
$$
where $c_k$ is the descending $k$-cycle $(k, k-1, \dotsc, 1)$ (note that we compose the elements of $\sym{s}$ from right to left).  This module is also the image under the Schur functor of the Lie power $L^s(V)$, which is the homogeneous part of degree $s$ of the free Lie algebra on an $n$-dimensional vector space $V$ with $n \geq s$.

In our study, we found that the knowledge of the image of $L^s(M)$ (where $M$ is a general $S(n,r)$-module; note that $V$ is naturally an $S(n,1)$-module) under the Schur functor will be most useful.  For example, the formula we provide here is used by Bryant and Erdmann \cite{BE} to understand the summands of $\Lie(s)$ lying in non-principal blocks.  It is also used by Erdmann and the authors \cite{ELT} in their study of the complexity of $\Lie(s)$.

As we shall see, the image of $L^s(M)$ under the Schur functor can be easily obtained as a corollary by understanding the image under the Schur functor of the tensor power $M^{\otimes s}$ as a $(F\sym{rs},F\sym{s})$-bimodule.  The latter result can be regarded as a refinement of a special case of \cite[2.5, Lemma]{DE}, although our proof is independent of their result.  Besides obtaining the image of $L^s(M)$ under the Schur functor from this refinement, we can also get those of the symmetric powers $S^s(M)$ and exterior powers $\E^s(M)$.  Our proofs are fairly elementary.

The organisation of the paper is as follows:  in the next section, we give a background on Schur algebras and a summary of the results we need.  We then proceed in Section \ref{S:main} to state and prove our main results.

\section{Schur algebras}

We briefly discuss Schur algebras and the results we need in this section.  The reader may refer to \cite{G} for more details.

Throughout, we fix an infinite field $F$ of arbitrary characteristic.

Let $n,r \in \mathbb{Z}^+$.  The Schur algebra $S(n,r)$ has a distinguished set $\{ \xi_{\alpha} \mid \alpha \in \Lambda(n,r)\}$ of pairwise orthogonal idempotents which sum to 1, where $\Lambda(n,r)$ is the set of compositions of $r$ with $n$ parts \cite[(2.3d)]{G}.  Thus each left $S(n,r)$-module $M$ has a vector space decomposition
$$
M = \bigoplus_{\alpha \in \Lambda(n,r)} \xi_\alpha M.
$$
We write $M^{\alpha}$ for $\xi_{\alpha}M$, and call it the $\alpha$-weight space of $M$.

Let $\GL_n(F)$ be the general linear group.  There is a surjective algebra homomorphism $e_r : F\GL_n(F) \to S(n,r)$ \cite[(2.4b)(i)]{G}.
If $s$ is another positive integer, and $M_1$ and $M_2$ are left $S(n,r)$- and $S(n,s)$-modules respectively, then $M_1 \otimes_F M_2$ can be endowed with a natural left $S(n,r+s)$-module structure, which satisfies
\begin{equation*} \label{E:tensor}
e_{r+s}(g) (m_1 \otimes m_2) =  (e_r(g)m_1) \otimes (e_s(g)m_2)
\end{equation*}
for all $g \in \GL_n(F)$, $m_1 \in M_1$ and $m_2 \in M_2$.  The weight spaces of $M_1 \otimes_F M_2$ can be described \cite[(3.3c)]{G} in terms of the weight spaces of $M_1$ and $M_2$, as follows:
\begin{equation} \label{E:weight}
(M_1 \otimes_F M_2)^{\gamma} = \bigoplus_{\substack{\alpha \in \Lambda(n,r) \\ \beta \in \Lambda(n,s) \\ \alpha+\beta = \gamma}} M_1^{\alpha} \otimes_F M_2^{\beta}.
\end{equation}
(Here, and hereafter, if $\alpha = (\alpha_1,\dotsc, \alpha_n) \in \Lambda(n,r)$ and $\beta = (\beta_1,\dotsc, \beta_n) \in \Lambda(n,s)$, then $\alpha + \beta = (\alpha_1 + \beta_1,\dotsc, \alpha_n + \beta_n) \in \Lambda(n,r+s)$.)

The symmetric group $\sym{n}$ on $n$ letters acts on $\Lambda(n,r)$ by place permutation:
$\tau \cdot (\alpha_1,\dotsc, \alpha_n) = (\alpha_{\tau^{-1}(1)},\dotsc, \alpha_{\tau^{-1}(n)})$.
We also view $\sym{n}$ as the subgroup of $\GL_n(F)$ consisting of permutation matrices.  Thus, $\sym{n}$ also acts naturally on left $S(n,r)$-modules via $e_r$.  We have the following lemma:

\begin{lem}\label{L:iso}
Let $n,r \in \mathbb{Z}^+$, and let $M$ be a left $S(n,r)$-module.
Let $\sigma \in \sym{n}$ and $\alpha = (\alpha_1,\dotsc, \alpha_n) \in \Lambda(n,r)$.
\begin{enumerate}
\item[(i)] $e_r(\sigma)$ maps $M^{\alpha}$ bijectively onto $M^{\sigma \cdot \alpha}$.
\item[(ii)] If $\sigma(i) = i$ for all $i$ such that $\alpha_i \ne 0$, then $e_r(\sigma)$ acts as identity on $M^{\alpha}$.  (Equivalently, if $\sigma_1(i) = \sigma_2(i)$ for all $i$ such that $\alpha_i \ne 0$, then $e_r(\sigma_1)m = e_r(\sigma_2)m$ for all $m \in M^{\alpha}$.)
\end{enumerate}
\end{lem}

\begin{proof}
  Part (i) is (3.3a) of \cite{G} (and its proof).  For part (ii), it follows from the definition of $e_r$ in \cite[\S2.4]{G} that $e_r(\sigma) m = \xi_{\sigma\mathbf{i},\mathbf{i}}m$ for $m$ lying in a weight space associated to $\mathbf{i}$, so that $e_r(\sigma) m = m$ when $m \in M^{\alpha}$, and $\sigma$ satisfies the condition in (ii).
\end{proof}

In the case where $n \geq r$, let $$\omega_r = (\underbrace{1,\dotsc, 1}_{r \text{ times}}, \underbrace{0,\dotsc, 0}_{n-r \text{ times}}) \in \Lambda(n,r),$$
The subalgebra $\xi_{\omega_r} S(n,r) \xi_{\omega_r}$ of $S(n,r)$ is isomorphic to $F\sym{r}$ \cite[(6.1d)]{G}.  This induces the Schur functor $f_r : {}_{S(n,r)}\textbf{mod} \to {}_{F\sym{r}}\textbf{mod}$ which sends a left $S(n,r)$-module $M$ to its weight space $M^{\omega_r}$.  The $\sym{r}$-action on $f_rM = M^{\omega_r}$ is that via $e_r$ and viewing $\sym{r}$ as a subgroup of $\GL_n(F)$ via the embedding $\sym{r} \subseteq \sym{n} \subseteq \GL_n(F)$, i.e. if $m \in f_rM$ and $\sigma \in \sym{r}$, then
\begin{equation*} \label{E:symaction}
\sigma \cdot m = e_r(\sigma)  m.
\end{equation*}

\section{Main results} \label{S:main}

Let $M$ be a left $S(n,r)$-module, and let $s \in \mathbb{Z}^+$.  The $s$-fold tensor product $M^{\otimes s}$ is then a left $S(n,rs)$-module, and it also admits another commuting right action of $\sym{s}$ by place permutations, i.e.\ $(m_1 \otimes \dotsb \otimes m_s) \cdot \sigma = m_{\sigma(1)} \otimes \dotsb \otimes m_{\sigma(s)}$ where $m_1,\dotsc, m_s \in M$, $\sigma \in \sym{s}$. As such, if $n \geq rs$, then $f_{rs} M^{\otimes s}$ is a $(F\sym{rs}, F\sym{s})$-bimodule.

On the other hand, $(f_r M)^{\otimes s}$ is a left $F(\sym{r} \wr \sym{s})$-module via
$$ (\sigma_1,\dotsc,\sigma_s) \tau \cdot (m_1 \otimes \dotsb \otimes m_s) = \sigma_1 m_{\tau^{-1}(1)} \otimes \dotsb \otimes \sigma_s m_{\tau^{-1}(s)}$$ and we can make it into a $(F(\sym{r} \wr \sym{s}), F\sym{s})$-bimodule by allowing $\sym{s}$ to act trivially on its right, while $F\sym{s}$ is naturally a $(F\sym{s},F\sym{s})$-bimodule and we can make it into a $(F(\sym{r} \wr \sym{s}), F\sym{s})$-bimodule by allowing $(\sym{r})^s$ to act trivially on its left.  Thus, $(f_r M)^{\otimes s} \otimes_F F\sym{s}$ is a $(F(\sym{r} \wr \sym{s}), F\sym{s})$-bimodule via the diagonal action.

For each $1\leq i\leq s$ and $\sigma\in \sym{r}$, we write $\sigma[i]\in \sym{rs}$ for the permutation sending $(i-1)r+j$ to $(i-1)r+\sigma(j)$ for each $1\leq j\leq r$, and fixing everything else pointwise; also, let $\sym{r}[i] = \{ \sigma[i] \mid \sigma \in \sym{r} \}$.  For $\tau\in \sym{s}$, we write $\tau^{[r]}\in\sym{rs}$ for the permutation sending $(i-1)r+j$ to $(\tau(i)-1)r+j$ for each $1\leq i\leq s$ and $1\leq j \leq r$; also, let $\sym{s}^{[r]} = \{ \tau^{[r]} \mid \tau \in \sym{s} \}$. We identify $\sym{r}\wr \sym{s}$ with the subgroup $(\prod_{i=1}^s \sym{r} [i])\sym{s}^{[r]}$ of $\sym{rs}$.

With the above understanding, we have the following result.

\begin{thm} \label{T:main}
Let $n, r,s \in \mathbb{Z}^+$ with $n \geq rs$, and let $M$ be an $S(n,r)$-module.  Then
$$
f_{rs} M^{\otimes s} \cong \Ind_{\sym{r} \wr \sym{s}}^{\sym{rs}} ((f_r M)^{\otimes s} \otimes_F F\sym{s})$$ as $(F\sym{rs}, F\sym{s})$-bimodules.
\end{thm}

\begin{proof}
Firstly, by \eqref{E:weight},
\begin{equation} \label{E:weight2}
f_{rs} M^{\otimes s} = (M^{\otimes s})^{\omega_{rs}} = \bigoplus_{(\alpha^{[1]}, \dotsc, \alpha^{[s]}) \in \Lambda} M^{\alpha^{[1]}} \otimes_F \dotsb \otimes_F M^{\alpha^{[s]}},
\end{equation}
where $\Lambda = \{ (\alpha^{[1]}, \dotsc, \alpha^{[s]}) \mid \alpha^{[i]} \in \Lambda(n,r)\ \forall i,\ \sum_{i=1}^s \alpha^{[i]} = \omega_{rs} \}$.
Also,
\begin{align*}
\Ind_{\sym{r} \wr \sym{s}}^{\sym{rs}} ((f_r M)^{\otimes s} \otimes_F F\sym{s}) &= F\sym{rs} \otimes_{F(\sym{r} \wr \sym{s})} ((f_r M)^{\otimes s} \otimes_F F\sym{s}) \\
&= \bigoplus_{t \in T} t \otimes ((f_r M)^{\otimes s} \otimes 1),
\end{align*}
where $T$ is a fixed set of left coset representatives of $\prod_{i=1}^s \sym{r}[i]$ in $\sym{rs}$.

The symmetric group $\sym{rs}$ acts naturally and transitively on the set
$$\Omega = \left\{ (A_1,\dotsc, A_s) \mid |A_i| = r\ \forall i,\ \bigcup_{i=1}^s A_i = \{ 1,\dotsc, rs\} \right\},$$
with $\prod_{i=1}^s \sym{r}[i]$ being the stabiliser of $(\{1,\dotsc, r\},\dotsc, \{(s-1)r + 1,\dotsc, rs\})$.  As such, the function $\theta : T \to \Omega$ defined by
$$t \mapsto (\{t(1),\dotsc, t(r)\},\dotsc, \{t((s-1)r+1), \dotsc, t(rs)\})$$ is a bijection.

On the other hand, each $r$-element subset $A$ of $\{1,\dotsc, rs\}$ corresponds naturally to a distinct element $\alpha_A = ((\alpha_A)_1,\dotsc, (\alpha_A)_n) \in \Lambda(n,r)$ defined by $(\alpha_A)_j = 1$ if $j \in A$, and $0$ otherwise.  This induces a bijection $\chi: \Omega \to \Lambda$ defined by
$$
(A_1,\dotsc, A_s) \mapsto (\alpha_{A_1}, \dotsc, \alpha_{A_s}).$$

For each $t \in T$ and $i = 1,\dotsc, s$, let $\tau_{t,i}$ be any fixed element of $\sym{rs}$ satisfying $\tau_{t,i}(j) = t((i-1)r + j)$ for $j = 1, \dotsc, r$ (we shall see below that how $\tau_{t,i}$ acts on other points is immaterial for our purposes).  Let $\alpha^{[i]} = \tau_{t,i} \cdot \omega_r$; then $\alpha^{[i]} = \alpha_{A_i}$, where
$$
A_i = \{ \tau_{t,i}(1), \dotsc, \tau_{t,i}(r) \} = \{ t((i-1)r + 1), \dotsc, t(ir) \}.
$$
Thus $(\alpha^{[1]},\dotsc, \alpha^{[s]}) = \chi(\theta (t)) \in \Lambda$.
Let
\begin{align*}
\phi_t : t \otimes ((f_r M)^{\otimes s} \otimes 1) &\to M^{\alpha^{[1]}} \otimes_F \dotsb \otimes_F M^{\alpha^{[s]}} \\
t \otimes ((x_1 \otimes \dotsb \otimes x_s) \otimes 1) &\mapsto e_r(\tau_{t,1}) x_1 \otimes \dotsb \otimes e_r(\tau_{t,s}) x_s \quad (x_1, \dotsc, x_s \in f_rM).
\end{align*}
Since each $\tau_{t,i}$ maps $f_r M = M^{\omega_r}$ bijectively onto $M^{\tau_{t,i} \cdot \omega_r} = M^{\alpha^{[i]}}$ by Lemma \ref{L:iso}(i), we see that $\phi_t$ is bijective.  Now let
$$\phi = \bigoplus_{t \in T} \phi_t : \Ind_{\sym{r} \wr \sym{s}}^{\sym{rs}} ((f_r M)^{\otimes s} \otimes_F F\sym{s}) \to f_{rs} M^{\otimes s}.$$  This is well-defined and bijective by \eqref{E:weight2} (note that $\chi \circ \theta : T \to \Lambda$ is a bijection).

Let $g \in \sym{rs}$ and $h \in \sym{s}$.  Then if $t \in T$ and $x_1,\dotsc, x_s \in f_rM$, we have
\begin{align*}
g (t \otimes ((x_1 \otimes \dotsb \otimes x_s) \otimes 1)) h &=
gth^{[r]} \otimes ((x_{h(1)} \otimes \dotsb \otimes x_{h(s)}) \otimes 1) \\
&= t' \otimes ((e_r(g_1)x_{h(1)} \otimes \dotsb \otimes e_r(g_s)x_{h(s)}) \otimes 1),
\end{align*}
where $gth^{[r]} = t' \prod_{i=1}^s g_i[i]$ with $t' \in T$ and $g_i \in \sym{r}$ for all $i$.  Thus it is sent by $\phi$ to $e_r(\tau_{t',1})e_r(g_1) x_{h(1)} \otimes \dotsb \otimes e_r(\tau_{t',s})e_r(g_s) x_{h(s)}$.  Note that
\begin{multline*}
\tau_{t',i}(j) = t'((i-1)r + j) = gth^{[r]} (\prod_{i=1}^s g_i[i])^{-1} ((i-1)r + j) \\
= gth^{[r]} ((i-1) + g_i^{-1}(j)) = gt((h(i)-1) + g_i^{-1}(j)) = g\tau_{t,h(i)}g_i^{-1}(j),
\end{multline*}
so that $e_r(\tau_{t',i})e_r(g_i) x_{h(i)} = e_r(g\tau_{t,h(i)}g_i^{-1})e_r(g_i) x_{h(i)}$ by Lemma \ref{L:iso}(ii). Hence,
\begin{align*}
\phi(g (t \otimes ((x_1 \otimes \dotsb \otimes x_s) \otimes 1)) h) &=
 e_r(g) e_r(\tau_{t,h(1)}) x_{h(1)} \otimes \dotsb \otimes e_r(g) e_r(\tau_{t,h(s)}) x_{h(s)} \\
&= e_{rs}(g) (e_r(\tau_{t,1})x_1 \otimes \dotsb \otimes e_r(\tau_{t,s})x_s) h \\
&= g(\phi(t \otimes ((x_1 \otimes \dotsb \otimes x_s) \otimes 1))) h.
\end{align*}
Thus $\phi$ is an $(F\sym{rs},F\sym{s})$-bimodule isomorphism.
\end{proof}

The $s$-th Lie power $L^s(M)$, the $s$-th exterior power $\E^s(M)$ and the $s$-th symmetric power $S^s(M)$ of the left $S(n,r)$-module $M$ may be defined as follows:
\begin{align*}
L^s(M) &= (M^{\otimes s}) \upsilon_s; \\
\E^s(M) &= (M^{\otimes s}) (\sum_{\sigma \in \sym{s}} \sgn(\sigma) \sigma); \\
S^s(M) &= M^{\otimes s} \otimes_{F\sym{s}} F.
\end{align*}
Here, $\upsilon_s$ is the Dynkin-Specht-Wever element mentioned in Section \ref{S:intro}, and $\sgn$ is the signature representation of $\sym{s}$.

\begin{cor} \label{C:isom}
Let $n, r,s \in \mathbb{Z}^+$ with $n \geq rs$, and $M$ be an $S(n,r)$-module.  Then
\begin{align*}
f_{rs} L^s(M) &\cong \Ind_{\sym{r} \wr \sym{s}}^{\sym{rs}} ((f_r M)^{\otimes s} \otimes_F \Lie(s)); \\
f_{rs} \E^s(M) &\cong \Ind_{\sym{r} \wr \sym{s}}^{\sym{rs}} ((f_r M)^{\otimes s} \otimes_F \sgn); \\
f_{rs} S^s(M) &\cong \Ind_{\sym{r} \wr \sym{s}}^{\sym{rs}} (f_r M)^{\otimes s}
\end{align*}
as left $F\sym{rs}$-modules.
\end{cor}

\begin{proof}
Post-multiply $\upsilon_s$ and $\sum_{\sigma \in \sym{s}} \sgn(\sigma) \sigma$ to both sides of the isomorphism in Theorem \ref{T:main} to obtain the first two isomorphisms.  The third isomorphism is obtained by taking tensor product with $F$ over $F\sym{s}$ on the right of both sides of the same isomorphism.
\end{proof}

\end{document}